\documentclass[12pt,leqno]{amsart}
\usepackage{amssymb,amsmath,amsthm,amsfonts,mathtools,lmodern}

\usepackage[T1]{fontenc}
\usepackage[utf8]{inputenc}
\usepackage[english]{babel}
\usepackage[a4paper,top=2cm,bottom=2cm,left=2cm,right=2cm]{geometry}
\usepackage{graphicx}
\usepackage{braket}
\usepackage[autostyle]{csquotes}
\usepackage[capitalise]{cleveref}
\usepackage{enumerate}
\usepackage{fancyhdr}

\DeclareMathOperator{\Ker}{Ker}
\DeclareMathOperator{\Id}{\mathbf{Id}}
\DeclareMathOperator{\End}{End}
 \DeclareMathOperator{\ad}{ad}

\DeclareMathOperator{\Lie}{Lie}

\newcommand{\field}[1]{\mathbb{#1}}
\newcommand{\G}{\field{G}}                      % Carnot group
\newcommand{\F}{\field{F}}                      % Free Carnot group
\newcommand{\K}{\field{K}}
\newcommand{\N}{\field{N}}                      % Naturals
\newcommand{\R}{\field{R}}                      % Reals
\newcommand{\Sph}{\field{S}}                    % Circle and spheres
\newcommand{\Heis}{\field{H}}                    % more Heis group
\newcommand{\rank}{\mathrm{rank}}
\newcommand{\order}{\mathrm{order}}

\newcommand{\op}{{\scriptstyle{op}}}

\newcommand{\ff}{{\mathfrak f}}

\newcommand{\fg}{{\mathfrak g}}

\newcommand{\cJ}{{\mathcal J}}

\newcommand{\fso}{{\mathfrak so}}

\newcommand{\fv}{{\mathfrak v}}

\newcommand{\bA}{{\mathbf A}}

\newcommand{\boldI}{{\mathbf I}}
\newcommand{\boldJ}{{\mathbf J}}

\newcommand{\boldO}{{\mathbf O}}

\newcommand{\bb}{{\mathbf b}}

\newcommand{\bv}{{\mathbf v}}

\newcommand{\bw}{{\mathbf w}}
\newcommand{\sk}{\mathrm{Skew}}
\newcommand{\twowedge}{\textstyle{\bigwedge^2}}

\def\Barint_#1{\mathchoice
	{\mathop{\vrule width 6pt height 3 pt depth -2.5pt
			\kern -8pt \intop}\nolimits_{#1}}%
	{\mathop{\vrule width 5pt height 3 pt depth -2.6pt
			\kern -6pt \intop}\nolimits_{#1}}%
	{\mathop{\vrule width 5pt height 3 pt depth -2.6pt
			\kern -6pt \intop}\nolimits_{#1}}%
	{\mathop{\vrule width 5pt height 3 pt depth -2.6pt
			\kern -6pt \intop}\nolimits_{#1}}}

\theoremstyle{plain}
\newtheorem{theorem}{Theorem}
\newtheorem{corollary}[theorem]{Corollary}
\newtheorem{lemma}[theorem]{Lemma}
\newtheorem{proposition}[theorem]{Proposition}

\theoremstyle{definition}
\newtheorem{definition}[theorem]{Definition}
\newtheorem{example}[theorem]{Example}
\newtheorem{remark}[theorem]{Remark}

\numberwithin{theorem}{section} \numberwithin{equation}{section}

\title{On the H-type deviation of step two Carnot groups}
\author{Luca Nalon}
\author{Jeremy T. Tyson}
\address{LN: D\'epartement de Math\'ematiques \\ Université de Fribourg \\ Ch. du mus\'ee 23 \\ 1700 Fribourg (CH) \\ \texttt{luca.nalon@unifr.ch}}
\address{JTT: Department of Mathematics \\ University of Illinois at Urbana-Champaign \\ 1409 West Green Street \\ Urbana, IL 61801 \\ \texttt{tyson@illinois.edu}}
\thanks{The first author was partially supported by the
		Swiss National Science Foundation (grant 200021-204501 ‘Regularity of sub-Riemannian
		geodesics and applications’).}
\thanks{The second author acknowledges support from the Simons Foundation under grant \#852888, `Geometric mapping theory and geometric measure theory in sub-Riemannian and metric spaces'. In addition, this material is based upon work supported by and while the second author was serving as a Program Director at the National Science Foundation. Any opinion, findings, and conclusions or recommendations expressed in this material are those of the author and do not necessarily reflect the views of the National Science Foundation.}
\date{\today}

\begin{document}
\maketitle
	
\begin{abstract}
The H-type deviation, $\delta(\G)$, of a step two Carnot group $\G$ quantifies the extent to which $\G$ deviates from the geometrically and algebraically tractable class of Heisenberg-type (H-type) groups. In an earlier paper, the second author defined this notion and used it to provide new analytic characterizations for the class of H-type groups. In this paper, we elucidate further properties of the H-type deviation. We investigate the behavior of $\delta(\G)$ under Carnot group morphisms and show that the step two and rank $m$ free Carnot group maximizes the H-type deviation among all step two and rank $m$ groups. We relate the value of the H-type deviation of $\G$ to the size of its abnormal set, via the concept of the {\it M\'etivier order} of a step two stratified Lie algebra. In this context we observe the distinguished role played by the so-called M\'etivier Carnot groups. We establish a rigidity theorem for the H-type deviation: for each $m \ge 3$ there exists $\delta_0(m)>0$ so that if $\G$ is a step two and rank $m$ Carnot group with $\delta(\G) < \delta_0(m)$, then $\G$ is a M\'etivier group. We use the optimality of the free step two group to conclude that each step two Carnot group admits a vertical metric which realizes the value of the H-type deviation. Finally, we compute the H-type deviation of direct products of step two Carnot groups.
\end{abstract}

\tableofcontents

\section{Introduction}

This paper is a sequel to \cite{tys:h-type-stability}, which introduced the {\it H-type deviation} $\delta(\G)$ of a step two Carnot group $\G$. The H-type deviation quantifies the extent to which $\G$ deviates from the class of Heisenberg-type groups. Heisenberg-type (or H-type) groups, introduced by Kaplan in \cite{kap:h-type}, are a geometrically and algebraically well-behaved family of groups distinguished by the validity of a host of explicit formulas for well-known analytic quantities such as the fundamental solution for the Folland sub-Laplacian and, more generally, the corresponding quasilinear sub-$p$-Laplacian. H-type groups have featured prominently in numerous papers exploring analysis and geometry in sub-Riemannian Carnot groups. A partial (and woefully incomplete) list of references is \cite{cdg:carnot}, \cite{rig:mass-transportation-type-H}, \cite{mv:quasimeromorphic-type-H}, \cite{eld:H-type-1}, \cite{eld:H-type-2}, \cite{bg:CDK}, and \cite{br:MCP}.

Motivation for the concept of H-type deviation arose from efforts to characterize the class of {\it polarizable Carnot groups}. Introduced by the second author and Balogh in \cite{bt:polar}, this class was defined by assuming the existence of an explicit one-parameter family of fundamental solutions for the $p$-Laplacian operators in terms of Folland's fundamental solution. All H-type groups are polarizable, and conjecturally no other polarizable groups exist. Each polarizable group carries an essential foliation by horizontal curves which allows for the computation of Haar integrals via a radial-spherical decomposition with horizontal radial curves. Such {\it horizontal polar coordinate} decomposition was used in \cite{bt:polar} to compute explicit formulas for moduli of ring domains and explicit constants in several geometric functional inequalities. The recent paper \cite{tys:polar2} establishes that the validity of a horizontal polar coordinate integration formula characterizes the class of polarizable Carnot groups. 

In \cite{tys:h-type-stability}, a quantitative conjecture was formulated relating the H-type deviation of a step two Carnot group $\G$ to the behavior of the $\infty$-Laplacian of Folland's fundamental solution for the $2$-Laplacian. An affirmative answer to this conjecture would imply that all step two polarizable groups are of Heisenberg type. Among the results in \cite{tys:h-type-stability} was a verification of this conjecture on a certain class of anistropic Heisenberg groups.

This paper advances the study of the H-type deviation with several new results. We first investigate the behavior of $\delta(\G)$ under Carnot group morphisms. In particular, we prove the following result relating the H-type deviation of an arbitrary step two group of rank $m$ with its free group counterpart.

\begin{theorem}\label{th:deviation-bounds}
Let $\F_{2,m}$ be the step two free Carnot group of rank $m$, with Lie algebra $\ff_m = \R^m \oplus \twowedge \R^m$. Then, for every step two Carnot group $\G$ of rank $m$, we have $\delta(\F_{2,m}) \ge \delta(\G)$.
\end{theorem}

Theorem \ref{th:deviation-bounds} is a special case of a more general result comparing the H-type deviations of a pair of step two groups related by a Carnot group epimorphism, see Proposition \ref{deviation_quotient}. As a corollary of Theorem \ref{th:deviation-bounds}, and in view of the value $\delta(\F_{2,m}) = \sqrt{(m-2)/m}$ computed in \cite[Example 3.6]{tys:h-type-stability}, we obtain the sharp upper bound for the H-type deviation of step two rank $m$ groups.

\begin{corollary}\label{cor:deviation-bounds}
Let $\G$ be a step two Carnot group of rank $m$. Then
\begin{equation}\label{eq:deviation-bounds}
0 \le \delta(\G) \le \sqrt{\frac{m-2}{m}}.
\end{equation}
\end{corollary}

When considering the H-type deviation, a distinguished role is played by the so-called M\'etivier Carnot groups. This class of groups (see Definition \ref{def:metivier}) includes all groups of Heisenberg type. We establish the following rigidity theorem for the H-type deviation.

\begin{theorem}\label{th:h-type-rigidity}
Let $\G$ be a step two and rank $m$ Carnot group.
\begin{itemize}
\item[(a)] If $m$ is even and $\delta(\G) < \sqrt{2/m}$, then $\G$ is a M\'etivier group.
\item[(b)] If $m$ is odd, then $\delta(\G) \ge \sqrt{1/m}$.
\end{itemize}
\end{theorem}

The value $\sqrt{2/m}$ is best possible in part (a), cf.\ Example \ref{ex:rigidity-sharpness}. Note that there exist M\'etivier groups which are not groups of Heisenberg type. The dichotomy between even and odd rank is natural here, since all M\'etivier groups have even rank. Coupling Corollary \ref{cor:deviation-bounds} and Theorem \ref{th:h-type-rigidity}(b), we conclude that $\delta(\G) = 1/\sqrt{3}$ for each step two and rank three Carnot group $\G$. An example of such a group is the free step two and rank three group.

To further indicate the connection between M\'etivier groups and the H-type deviation, we introduce the concept of {\it M\'etivier order} of a step two Lie algebra $\fg$. This quantity, denoted $\order(\fg)$, is defined as the minimal rank of Kaplan operators associated to nonzero elements in the dual space of the second layer of $\fg$. It is straightforward to observe that a Carnot group $\G$ is a M\'etevier group if and only if its M\'etivier order is equal to its rank. At the opposite extreme lie those step two Carnot groups of M\'etivier order two. We observe that the latter class of groups consists precisely of those groups which realize the upper bound in Corollary~\ref{cor:deviation-bounds}.

\begin{corollary}\label{cor:deviation-bounds-optimality}
Let $\G$ be a step two Carnot group of rank $m$, with Lie algebra $\fg$. Then
$$
\delta(\G) = \sqrt{\frac{m-2}{m}}
$$
if and only if the M\'etivier order of $\fg$ is equal to two.
\end{corollary}

Finally, we calculate the H-type deviation of direct products of step two Carnot groups. Specifically, we prove the following theorem.

\begin{theorem}\label{th:products-h-type-deviation}
Let $\G_1,\ldots,\G_\ell$ be step two Carnot groups. Then
\begin{equation}\label{eq:products-h-type-deviation}
	\delta(\G_1 \times \cdots \times \G_\ell)= \sqrt{1- \frac{\min\set{\rank(\G_i)(1-\delta^2(\G_i)):i=1,\ldots,\ell}}{\rank(\G_1\times\cdots\times\G_\ell)}}.
\end{equation}
Moreover,
\begin{equation}\label{eq:products-h-type-deviation-with-Euclidean-factors}
	\delta(\G_1 \times \cdots \times \G_\ell \times \R^\nu)= \sqrt{1- \frac{\min\set{\rank(\G_i)(1-\delta^2(\G_i)):i=1,\ldots,\ell}}{\rank(\G_1\times\cdots\times\G_\ell\times\R^\nu)}}.
\end{equation}\end{theorem}

In \eqref{eq:products-h-type-deviation-with-Euclidean-factors} the Euclidean factor $\R^\nu$ is viewed as a component of the first (horizontal) layer of the product Carnot group, see Remark \ref{rem:Euclidean-factors} for further details.

\smallskip

We conclude this introduction with an outline of the paper. 

\smallskip

Section \ref{sec:background} contains background information and preliminary definitions. In section \ref{sec:morphism} we study the morphism properties of the H-type deviation and prove Theorem \ref{th:deviation-bounds} and Corollaries \ref{cor:deviation-bounds} and \ref{cor:deviation-bounds-optimality}. In section \ref{sec:rigidity} we prove the rigidity result Theorem \ref{th:h-type-rigidity} and comment further on the structure of groups with small H-type deviation. In section \ref{sec:realization} we show that the infimum in the definition of the H-type deviation $\delta(\G)$ is always attained by some particular vertical metric. Such a conclusion was claimed in \cite{tys:h-type-stability} in the case $\delta(\G) = 0$, where it played a role in the proof that groups of vanishing H-type deviation are (nascent) H-type groups. Our discussion in section \ref{sec:realization} clarifies one unclear point in that previous argument, specifically, why the limiting semimetric defined on the second layer is actually a metric. We illustrate the subtle nature of this observation with an in-depth consideration of the first quaternionic Heisenberg group.

Finally, section \ref{sec:products} contains the proof of Theorem \ref{th:products-h-type-deviation}. As an example, we compute the H-type deviation of arbitrary products of (isotropic and anisotropic) Heisenberg groups and some related examples.

\smallskip

\paragraph{\bf Acknowledgements.} 
The second author expresses thanks to Nicola Garofalo for an informative conversation on the subject of this paper, and for helpful comments regarding the potential relevance of M\'etivier groups. Both authors wish to thank Enrico Le Donne for valuable comments, especially for useful information regarding the abnormal sets problem and the structure of M\'etivier groups. 

An earlier version of some results in this paper, written by the second author alone, was released in December 2023 ({\tt arxiv.org/abs/2312.06076}). The current paper supersedes and replaces {\tt arxiv.org/abs/2312.06076}. Several results are proved in stronger form than stated there, new and streamlined proofs are presented for several conclusions, and some new results are also obtained.

\section{Background and preliminary definitions}\label{sec:background}

\subsection{Step two Carnot groups} 

Let $\fg$ be an algebra, we say that a direct sum decomposition $\fg = V_1 \oplus V_2$ is a {\it step two stratification} if
\begin{equation}\label{eq:lie-bracket-relation}
	[V_1,V_1]=V_2,
\end{equation} 
$[V_1,V_2] = 0$, and $V_2 \neq 0$. A Lie algebra $\fg$ equipped with a step two stratification is called {\it step two stratified}, this implies that $\fg$ is nilpotent of step two. Conversely, every step two nilpotent Lie algebra admits a stratification, which is unique up to a Lie algebra automorphism. Given a step two stratified Lie algebra $\fg = V_1 \oplus V_2$, we refer to $\dim(V_1)$ as the {\it rank} of $\fg$ and denote it by $\rank(\fg)$. A {\it step two stratified group} is a connected, simply connected Lie group whose associated Lie algebra $\fg = \Lie(\G)$ is step two stratified.

An inner product $g_h$ defined in the first layer $V_1$ of a step two stratified Lie algebra is called a {\it horizontal metric} on $\fg$. It is well known that such a choice of metric on the Lie algebra of a step two stratified group $\G$ equips it with a sub-Riemannian structure. We refer to such groups equipped with a distinguished sub-Riemannian structure as {\it step two Carnot groups}.

\subsection{Kaplan's operator and the M\'etivier order of a step two algebra} 
If $V$ is a vector space, we denote by $\sk(V)$ the set of skew-symmetric bi-linear forms on $V$. Given a step two stratified Lie algebra $\fg = V_1 \oplus V_2$, with Lie bracket $[\cdot,\cdot]$, we define the Kaplan's operator $\boldJ$ as
\begin{equation} \label{eq:Kaplan}
	\boldJ \colon V_2^* \to \sk(V_1),  \qquad \mu \mapsto \boldJ_\mu ,
\end{equation}
where $\boldJ_\mu(\bv,\bw) \coloneq \mu([\bv,\bw])$. We recall that the map $\boldJ$ is injective. 

\begin{remark}\label{rem:step-two-g-as-a-quotient-of-free-step-two}
Let $\fg = V_1 \oplus V_2$ be a step two stratified Lie algebra as above. Then $V_2$ may be identified with a quotient of $\twowedge(V_1)$:
\begin{equation}\label{eq:Lie-alg-identified}
\fg \simeq V_1 \oplus \frac{\twowedge(V_1)}{\mathfrak{w}},
\end{equation}
for a suitable ideal $\mathfrak{w} \subset \twowedge(V_1)$. The Lie bracket relation on the right hand side of \eqref{eq:Lie-alg-identified} is given by $[\bv_1\oplus \bw_1 + \mathfrak{w},\bv_2 \oplus \bw_2 + \mathfrak{w}] = 0 \oplus \bv_1\wedge \bv_2 + \mathfrak{w}$. Under this identification, the Kaplan operator becomes the canonical correspondence between the dual of $\twowedge(V_1)$ and $\sk(V_1)$:
$$
\boldJ \, : \, (\twowedge(V_1))^* \to \sk(V_1), \qquad \boldJ_\mu(\bv,\bw) = \mu(\bv\wedge \bw).
$$
For a subspace $\mathfrak{w}$ as above, we identify $\twowedge(V_1) / \mathfrak{w}$ with the orthocomplement $\mathfrak{w}^\perp \subset (\twowedge(V_1))^*$.
\end{remark}

%%%%%% Since \mathfrak{w} is an ideal, maybe we can use \fw in thise case, since we never use cursive in the rest of the paper %%%%%

\begin{definition}
	Let $\fg$ be a step two stratified Lie algebra. We define the {\it M\'etivier order} of $\fg$, denoted $\order(\fg)$, to be the number
	\begin{equation*}
		\order(\fg) \coloneq \min \Set{\rank(\boldJ_\mu) : \mu \in V_2^*\setminus\set{0} }.
	\end{equation*}
\end{definition}

\begin{definition}\label{def:metivier}
	A step two stratified Lie algebra $\fg = V_1 \oplus V_2$ is called {\it Métivier} if the map
	\begin{equation*}
		\ad_\bv \colon V_1 \to V_2, \quad \bw \mapsto [\bv,\bw]
	\end{equation*}
	is surjective for every nonzero $\bv \in V_1$. A {\it Métivier group} is a step two stratified group whose Lie algebra is Métivier.
\end{definition}

We remark that the M\'etivier order of a step two stratified Lie algebra $\fg$ is always an even number between two and the rank of $\fg$. Moreover, $\rank(\fg)=\order(\fg)$ precisely when the Lie algebra is Métivier. We stress that the M\'etivier order of a step two stratified Lie algebra is invariant under the choice of stratification. Moreover, it is precisely the double of the invariant $\widetilde{k}$ appearing in \cite[Theorem~1.1]{bnv:SP2}. We can then state the following corollary of \cite[Theorem~1.1]{bnv:SP2}.

\begin{corollary}\label{cor:abnormal}
	Let $\G$ be a step two stratified group, with algebra $\fg$ of M\'etivier order $2k$. Then the abnormal set of $\G$ is contained in an algebraic variety of co-dimension at least $2k+1$.
\end{corollary}

\subsection{H-type deviation} We consider an horizontal metric $g_h$ defined on a step two stratified Lie algebra $\fg=V_1 \oplus V_2$. The inner product $g_h$ and the Kaplan's operator induce a map 
\begin{equation} \label{eq:Kaplan_map}
	\boldJ^{[g_h]} \colon V_2^* \to \End(V_1), \quad \mu \mapsto \boldJ^{[g_h]}_\mu,
\end{equation}
where $\End(V_1)$ denotes the space of linear self-maps of $V_1$ and $\boldJ^{[g_h]}$ is defined by $g_h(\boldJ_\mu^{[g_h]}(\bv),\bw)=\boldJ_\mu(\bv,\bw)$, for every $\bv,\bw \in V_1$.

\begin{definition}
	Let $\G$ be a step two Carnot group with stratified Lie algebra $\fg = V_1 \oplus V_2$ and horizontal metric $g_h$. We say that $\G$ is an {\it H-type (Heisenberg-type) group} if there exists an inner product $g_v$ on $V_2^*$ such that $\boldJ^{[g_h]}_\mu$ is an orthogonal self-map of $(V_1,g_h)$ for every $\mu \in V_2^*$ such that $\lvert \mu \rvert_v=1$.
\end{definition}

Although the notion of H-type group is formulated using a fixed inner product $g_v$ on the (dual of the) second layer, it can be verified for a given step two Carnot group solely in terms of the horizontal metric $g_h$. The following result may be known to experts in the area, however, we include the statement and a short proof here for the convenience of the reader.

\begin{proposition}
A step two Carnot group $\G$ is an H-type group if and only if the map $\boldJ^{[g_h]}_\mu$ is a conformal self-map of $(V_1,g_h)$ for every $0\ne \mu \in V_2^*$, i.e., $\boldJ^{[g_h]}_\mu=\lambda_\mu\bA_\mu$ for some orthogonal self-map $\bA_\mu$ of $(V_1,g_h)$ and some $\lambda_\mu > 0$.
\end{proposition}

\begin{proof}
The `only if' assertion is trivial. Towards the converse assertion, fix a step two Carnot group $\G$ with stratified Lie algebra $\fg = V_1 \oplus V_2$ and horizontal metric $g_h$. We emphasize that we do {\bf not} fix an inner product in the second layer $V_2^*$. Assume that $\boldJ^{[g_h]}_\mu$ is a conformal self-map of $(V_1,g_h)$ for every $0 \ne \mu \in V_2^*$. By \eqref{eq:Kaplan}, $\boldJ^{[g_h]}_\mu$ is also skew-symmetric, hence $(\boldJ^{[g_h]}_\mu)^2$ is a negative scalar multiple of the identity for every such $\mu$. We deduce that
\begin{equation*}
	\boldJ^{[g_h]}_\mu \boldJ^{[g_h]}_\nu + \boldJ^{[g_h]}_\nu \boldJ^{[g_h]}_\mu = \left[\left(\boldJ^{[g_h]}_{\mu+\nu}\right)^2 - \left(\boldJ^{[g_h]}_\mu\right)^2 - \left(\boldJ^{[g_h]}_\nu\right)^2\right]
\end{equation*}
is also a scalar multiple of the identity $\boldI \in \End(V_1)$ for every $\mu,\nu \in V_2^*$. We define $g_v(\mu,\nu)$ such that $\boldJ^{[g_h]}_\mu \boldJ^{[g_h]}_\nu + \boldJ^{[g_h]}_\nu \boldJ^{[g_h]}_\mu = -2g_v(\mu,\nu)\boldI$, for every $\mu,\nu \in V_2^*$. It is clear that $g_v$ is a positive-definite symmetric bi-linear form on $V_2^*$, and that $\boldJ^{[g_h]}_\mu$ is an orthogonal map whenever $g_v(\mu,\mu)=1$. Therefore $g_v$ equips $\G$ with the structure of an H-type group. 
\end{proof}

We stress that, for a step two Carnot group, the Métivier condition is necessary for the group to be of H-type. However, there exist Métivier groups that fail to be of H-type for any choice of horizontal metric (cf.\ the appendix in \cite{ms:maximal-operators}).

\begin{definition}
Let $\G$ be a step two Carnot group with stratified Lie algebra $\fg = V_1 \oplus V_2$ of rank $m$ and with horizontal metric $g_h$. For a given inner product $g_v$ on $V_2^*$ we define 
\begin{equation} \label{def:delta_gv}
	\delta(\G,g_v) \coloneq \frac{1}{\sqrt{m}} \sup \Set{\left\lVert \left(\boldJ^{[g_h]}_\mu\right)^2 + \boldI\right\rVert_{\mathrm{HS}} : \mu \in V_2^*, \, \lvert \mu \rvert_v=1}.
\end{equation}
We define the {\it H-type deviation} of $\G$ to be the number
$$
\delta(\G)\coloneq\inf_{g_v} \delta(\G,g_v),
$$
where the infimum is taken over all inner products $g_v$ on $V_2^*$.
\end{definition}

\begin{remark}
We have introduced the notion of H-type deviation slightly differently here than was done in \cite{tys:h-type-stability}, considering Kaplan's operator and vertical metrics on the dual space $V_2^*$ rather than on $V_2$ itself. A standard duality argument shows that the two concepts agree.
\end{remark}

\begin{proposition}\label{prop:order-vs-deviation}
	Let $\G$ be a step two Carnot group with stratified Lie algebra $\fg = V_1 \oplus V_2$ of rank $m$, M\'etivier order $2k$, and with horizontal metric $g_h$. Then
%, for every inner product $g_v$ on $V_2^*$, we have
	\begin{equation}\label{eq:H-type-deviation-vs-order}
		\delta(\G)\ge\sqrt{\frac{m-2k}{m}}.
	\end{equation} 
\end{proposition}

\begin{proof}
Fix an inner product $g_v$ on $V_2^*$. Consider $\mu \in V_2^*$ such that $\rank(\boldJ_\mu)=2k$. Since $\boldJ_\mu$ is a skew-symmetric bi-linear form, then the map $(\boldJ_\mu^{[g_h]})^2$ is symmetric with respect to $g_h$. Moreover $\rank(\boldJ_\mu^{[g_h]})^2=2k$ as well. We consider an orthonormal basis of $V_1$ for which $(\boldJ_\mu^{[g_h]})^2$ is diagonal. With respect to this basis, the map $(\boldJ^{[g_h]}_\mu)^2 + \boldI$ is again diagonal and its eigenspace corresponding to the eigenvalue $1$ has dimension $m-2k$. Therefore, we can infer that
	\begin{equation*}
	\left\lVert \left(\boldJ^{[g_h]}_\mu\right)^2 + \boldI \right\rVert_\mathrm{HS} \ge \sqrt{m-2k},
	\end{equation*}
and hence $\delta(\G,g_v)\ge\sqrt{\frac{m-2k}{m}}$. Infimizing over all metrics $g_v$ on $V_2^*$ and using \eqref{def:delta_gv} completes the proof.
\end{proof}

%\begin{remark}\label{rem:abnormal-to-h-type}
%The conclusion in Proposition \ref{prop:order-vs-deviation} may be restated in a way which illustrates how the H-type deviation controls the size of the abnormal set. Combining Corollary \ref{cor:abnormal} and Proposition \ref{prop:order-vs-deviation} yields the %following statement: {\it Let $\G$ be a step two stratified group, with algebra $\fg$ of rank $m$ and H-type deviation $\delta(\G)$. Then the abnormal set of $\G$ is contained in an algebraic variety of co-dimension at least
%\begin{equation}\label{eq:abnormal-to-h-type}
%3 + m \left( \frac{m-2}{m} - \delta(\G)^2 \right).
%\end{equation}}
%Note that the quantity in parentheses in \eqref{eq:abnormal-to-h-type} is nonnegative in view of Corollary \ref{cor:deviation-bounds}.
%\end{remark}

\begin{example}[Anisotropic Heisenberg groups]\label{ex:anisotropic-heisenberg}
Recall that the {\it anisotropic Heisenberg group} $\Heis^n(\bb)$, where $\bb = (b_1,\ldots,b_n)$ is a vector of positive real numbers, is the step two Carnot group of rank $2n$ and one-dimensional center for which
\begin{itemize}
\item the horizontal layer $V_1$ is spanned by a $g_h$-orthonormal basis $X_1,Y_1,\ldots,X_n,Y_n$,
\item the vertical layer $V_2$ is spanned by a single vector $T$, and
\item for each $j=1,\ldots,n$, we have $[X_j,Y_j] = b_j \, T$.
\end{itemize}
The case $\bb=(1,\ldots,1)$ is the (isotropic) $n$th Heisenberg group, which we denote by $\Heis^n$. For each choice of $\bb$ the group $\Heis^n(\bb)$ is isomorphic, as a step two stratified Lie group, to $\Heis^n$. On the other hand, two such groups $\Heis^n(\bb)$ and $\Heis^n(\bb')$ are isomorphic as Carnot groups if and only if $\bb$ and $\bb'$ are projectively equivalent, i.e., there exists $\lambda_0$ so that $\bb' = \lambda \, \bb$. In particular, $\Heis^n(\bb)$ is a Carnot group isomorphic to the standard Heisenberg group $\Heis^n$ if and only if $\bb$ is a constant vector.

In \cite[Example 3.4]{tys:h-type-stability} we computed the H-type deviation of $\Heis^n(\bb)$, obtaining
$$
\delta(\Heis^n(\bb)) = \sqrt{ 1 - \left( \frac{||\bb||_2}{||\bb||_4} \right)^4}.
$$
Here $||\bb||_p := \left( n^{-1} \sum_{j=1}^n b_j^p \right)^{1/p}$ is the normalized $\ell^p$ norm of the vector $\bb \in \R^n$.
\end{example}

\section{Morphism properties of the H-type deviation}\label{sec:morphism}

%\subsection{Quotient maps and step two free Carnot groups} 

We consider morphisms $\varphi \colon \fg \to \tilde{\fg}$ of step two stratified Lie algebras $\fg=V_1 \oplus V_2$ and $\tilde{\fg} = \widetilde{V}_1 \oplus \widetilde{V}_2$, equipped with horizontal metrics $g_h$ and $\tilde{g}_h$ respectively. If we assume that $\varphi_{|V_1}$ is an isometry from $(V_1,g_h)$ to $(\tilde{V}_1,\tilde{g}_h)$, then $\varphi$ is surjective and $\varphi(V_2)=\widetilde{V}_2$, indeed 
$$\varphi(V_2) = \varphi\left([V_1,V_1]\right) = [\left(\varphi(V_1)\right),\left(\varphi(V_1)\right)]=[\widetilde{V}_1,\widetilde{V}_1]=\widetilde{V}_2.$$
Therefore the dual map $\varphi^*_{|V_2}$ defines an inclusion 
\begin{equation} \label{eq:phi_inclusion}
	\varphi^*_{|V_2} \colon \widetilde{V}_2^* \xhookrightarrow{} V_2^*.
\end{equation}

We denote by $\boldJ$ and $\tilde{\boldJ}$ the Kaplan's operators of $\fg$ and $\tilde{\fg}$, respectively. Moreover, we denote by $\boldJ^{[g_h]}$ and $\tilde{\boldJ}^{[\tilde{g}_h]}$ the maps defined as in \eqref{eq:Kaplan_map}, respectively. We observe that, for every $\nu \in \tilde{V}_2^*$, the maps $\boldJ_{\varphi^*\nu}^{[g_h]}$ and $\tilde{\boldJ}_\nu^{[\tilde{g}_h]}$ are conjugated through $\varphi$, more precisely
\begin{equation*}
	\varphi \circ \boldJ_{\varphi^*\nu}^{[g_h]} = \tilde{\boldJ}_\nu^{[\tilde{g}_h]} \circ \varphi.
\end{equation*}
Indeed, for every $\bv,\bw \in V_1$, on the one hand
\begin{equation*}
	\tilde{g}_h\left(\varphi\left( \boldJ_{\varphi^*\nu}^{[g_h]}(\bv)\right),\varphi(\bw) \right)= g_h \left(\boldJ_{\varphi^*\nu}^{[g_h]}(\bv),\bw\right)= \boldJ_{\varphi^*\nu}(\bv,\bw)= \varphi^*\nu\left([\bv,\bw]\right)=\nu(\varphi[\bv,\bw]),
\end{equation*}
where we used that $\varphi_{|V_1}$ is an isometry and the definition of dual map. On the other hand
\begin{equation*}
	\tilde{g}_h\left( \tilde{\boldJ}_{\nu}^{[\tilde{g}_h]}\left(\varphi(\bv)\right),\varphi(\bw) \right)=\tilde{\boldJ}_\nu\left(\varphi(\bv),\varphi(\bw)\right)=\mu\left([\varphi(\bv),\varphi(\bw)]\right)= \mu(\varphi[\bv,\bw]),
	\end{equation*}
where we used that $\varphi$ is a Lie algebra morphism. Thus, since they are conjugated though an isometry, the maps $\boldJ_{\varphi^*\nu}^{[g_h]}$ and $\tilde{\boldJ}_\nu^{[\tilde{g}_h]}$ have the same Hilbert-Schmidt norm. 
  
\begin{proposition} \label{deviation_quotient}
Let $\G$ and $\widetilde{\G}$ be step two Carnot groups with horizontal metrics $g_h$ and $\tilde{g}_h$ and stratified Lie algebra $\fg = V_1 \oplus V_2$ and $\tilde{\fg} = \widetilde{V}_1 \oplus \widetilde{V}_2$, respectively. Assume that there exists a Lie algebra morphism $\varphi \colon \fg \to \tilde{\fg}$ such that $\varphi_{|V_1}$ is an isometry from $(V_1,g_h)$ to $(\widetilde{V}_1,\widetilde{g}_h)$. Then $\delta(\G) \ge \delta(\widetilde{\G})$.
\end{proposition}

\begin{proof}[Proof of Theorem \ref{deviation_quotient}]
	Fix an inner product $g_v$ on $V_2^*$. The map \eqref{eq:phi_inclusion} defines, by restriction, an inner product $\tilde{g}_v$ on $\widetilde{V}_2^*$. Therefore 
	\begin{align*}
		\delta(\G,g_v) &= \sup \Set{\left\lVert \left(\boldJ^{[g_h]}_\mu\right)^2 + \boldI\right\rVert_{\mathrm{HS}} : \mu \in V_2^*, \, \lvert \mu \rvert_v=1} \\ &\ge \sup \Set{\left\lVert  \left(\boldJ_{\varphi^*\nu}^{[g_h]}\right)^2 + \boldI \right\rVert_{\mathrm{HS}} : \nu \in \widetilde{V}_2^*, \, \lvert \nu \rvert_v=1} \\ &= \sup \Set{\left\lVert  \left(\tilde{\boldJ}_\nu^{[\tilde{g}_h]}\right)^2 + \boldI \right\rVert_{\mathrm{HS}} : \nu \in \widetilde{V}_2^*, \, \lvert \nu \rvert_v=1} = \delta(\widetilde{\G},\tilde{g}_v),
	\end{align*}
	where we used that $\lVert  (\boldJ_{\varphi^*\nu}^{[g_h]})^2 + \boldI \rVert_{\mathrm{HS}}=\lVert  (\tilde{\boldJ}_\nu^{[\tilde{g}_h]})^2 + \boldI \rVert_{\mathrm{HS}}$, since $\boldJ_{\varphi^*\nu}^{[g_h]}$ and $\tilde{\boldJ}_\nu^{[\tilde{g}_h]}$ have the same Hilbert-Schmidt norm. By taking the infimum over all possible metrics, we get the result.
\end{proof}

We define the {\it step two free nilpotent Lie algebra of rank $m$} as the Lie algebra
\begin{equation} \label{def:free_algebra}
	\ff_m \coloneq \R^m \oplus \twowedge \R^m,
\end{equation}
with Lie bracket $[v_1 + \omega_1,v_2 + \omega_2]=v_1 \wedge v_2$ for every $v_1,v_2 \in \R^m$ and $\omega_1,\omega_2 \in \twowedge \R^m$. We stress that \eqref{def:free_algebra} is a stratification of the Lie algebra $\ff_m$. We say that a Carnot group $\G$ is {\it step two free-Carnot of rank $m$} if its Lie algebra is isomorphic to $\ff_m$ 

It is well known that the step two free-nilpotent Lie algebra of rank $m$ has the following universal property:  

\begin{lemma} \label{free_extension}
Let $\fg = V_1 \oplus V_2$ be a step two stratified Lie algebra of rank $m$. Then every linear isomorphism $\varphi \colon \R^m \to V_1$ uniquely extends to a surjective Lie algebra morphism 
$$
\widetilde{\varphi} \colon \ff_m = \R^m \oplus \twowedge \R^m \to \fg.
$$
\end{lemma}

\begin{proof}[Proof of Theorem \ref{th:deviation-bounds}]
Let $\fg=V_1 \oplus V_2$ be the step two stratified Lie algebra of $\G$. Denote by $g_h$, respectively, $\tilde{g}_h$, the horizontal metric of $\F_{2,m}$ and $\G$. Fix an isometry $\varphi$ between $(\R^m,g_h)$ and $(V_1,\tilde{g}_h)$; such isometry exists since $\dim(V_1)=m$. By \eqref{free_extension}, the map $\varphi$ extends to a Lie algebra morphism $\widetilde{\varphi} \colon \ff_m \to {\fg}$. The map $\widetilde{\varphi}$ satisfies the hypotheses of \cref{deviation_quotient}, whence $\delta(\F_{2,m}) \ge \delta(\G)$.
\end{proof}

Corollary \ref{cor:deviation-bounds} follows from Theorem \ref{th:deviation-bounds} and the exact value for $\delta(F_{2,m})$, see \cite[Example 3.6]{tys:h-type-stability}. We take this opportunity to provide an alternate proof for Corollary \ref{cor:deviation-bounds}.

\begin{proof}[Second proof of Corollary \ref{cor:deviation-bounds}]
We introduce a canonical metric on $V_2^*$ via pullback of the Hilbert--Schmidt metric on $\fso(V_1,g_h)$. Define $g_v^\circ$ on $V_2^* \times V_2^*$ by
\begin{equation}\label{eq:gvo}
g_v^\circ(\mu,\nu) := \frac12 \langle \boldJ_\mu^{[g_h]},\boldJ_\nu^{[g_h]} \rangle_{HS}.
\end{equation}
Then
\begin{equation*}\begin{split}
\delta(\G)  \le \delta(\G,g_v^\circ) &= \frac1{\sqrt{m}} \sup_{\mu:||\mu||_{g_v^\circ} = 1} ||(\boldJ_\mu^{[g_h]})^2 + \boldI ||_{HS} \\
& = \frac1{\sqrt{m}} \sup_{\mu:||\boldJ_\mu||_{HS} = \sqrt2} ||\boldJ_\mu^2 + \boldI ||_{HS} \\
& = \frac1{\sqrt{m}} \sup_{\mu:||\boldJ_\mu||_{HS} = \sqrt2} \sqrt{||\boldJ_\mu^2||_{HS}^2 + 2 \langle \boldJ_\mu^2,\boldI \rangle_{HS} + m} \\
& = \frac1{\sqrt{m}} \sup_{\mu:||\boldJ_\mu||_{HS} = \sqrt2} \sqrt{||\boldJ_\mu^2||_{HS}^2 - 2 || \boldJ_\mu||_{HS}^2 + m}
\end{split}\end{equation*}
Using the inequality $||\bA^2||_{HS} \le \tfrac1{\sqrt{2}} ||\bA||_{HS}^2$, valid for all skew-symmetric matrices $\bA$ (see e.g. \cite[(2.3)]{tys:h-type-stability}), we conclude that $ \delta(\G) \le \sqrt{(m-2)/m}$ as desired.
\end{proof}

Next, we prove Corollary \ref{cor:deviation-bounds-optimality}, which characterizes those step two and rank $m$ Carnot groups which realize the upper bound in \eqref{eq:deviation-bounds}.

\begin{proof}[Proof of Corollary \ref{cor:deviation-bounds-optimality}]
Let $\G$ be a step two and rank $m$ Carnot group whose Lie algebra has M\'etivier order two. Proposition \ref{prop:order-vs-deviation} and Corollary \ref{cor:deviation-bounds} immediately imply that the H-type deviation of $\G$ is equal to $\sqrt{(m-2)/m}$.

For the converse direction, assume that $\delta(\G) = \sqrt{\tfrac{m-2}{m}}$. Then $\delta(\G,g_v^\circ) = \sqrt{\tfrac{m-2}{m}}$ for the metric $g_v^\circ$ considered in \eqref{eq:gvo}. Considering the second proof of Corollary \ref{cor:deviation-bounds} given above, we see that
$$
\frac1{\sqrt{m}}\sup_{\mu:||\boldJ_\mu||_{HS} = \sqrt{2}}\sqrt{m - 4 + ||\boldJ_\mu^2||_{HS}^2} = \sqrt{\tfrac{m-2}{m}}.
$$
This in turn implies that there exists $\mu \in \fv_2^*$ so that $||\boldJ_\mu||_{HS} = \sqrt{2}$ and $||\boldJ_\mu^2||_{HS}^2 = 2$, and in particular, 
$$
||\boldJ_\mu^2||_{HS}^2 = \tfrac12 ||\boldJ_\mu||_{HS}^4.
$$
Appealing to a standard block diagonal representation for skew-symmetric matrices (see, for instance, the proof of the estimates in \cite[(2.2)]{tys:h-type-stability}), choose a $g_h$-orthonormal basis $X_1,\ldots,X_m$ for $\fv_1$ so that $\mu([X_1,X_2]) > 0$ and $\mu([X_i,X_j]) = 0$ whenever $i+j>3$. Then $\rank(\boldJ_\mu) = 2$ and consequently $\fg$ has M\'etivier order two.
\end{proof}

Examples of rank $m$ groups $\G$ with M\'etivier order two include the free Carnot group of step $m$ as well as any product Carnot group of the form $\G = \K \times \Heis^1$, where $\K$ is a step two Carnot group and $\Heis^1$ is the first Heisenberg group. See Proposition \ref{prop:h-type-deviation-of-products-of-anisotropic-Heisenberg-groups-maximality-condition} for further detail on the latter example.

\begin{remark}
One may ask whether the metric $g_v^\circ$ defined in \eqref{eq:gvo} {\bf always} realizes the value of the H-type deviation, or alternatively always realizes the upper bound in Corollary \ref{cor:deviation-bounds}. Note that when $\G = \F_{2,m}$, then 
$$
\sqrt{\tfrac{m-2}{m}} = \delta(\G)\le \delta(\G,g_v^\circ) \le \sqrt{\tfrac{m-2}{m}}
$$
and hence equality holds throughout. 

However, neither coincidence holds in general. We illustrate this by discussing anisotropic Heisenberg groups. Let $\G = \Heis^n(\bb)$ with $n \ge 2$ and for some $n$-vector $\bb$. Let $T$ span the second layer $V_2$, and let $\mu \in V_2^*$ be dual to $T$. Recall \cite{tys:h-type-stability} that vertical metrics $g_v^\lambda$ on $V_2^*$ are parameterized by a positive real parameter $\lambda$ via the condition that $\lambda^{-1}\mu$ is $g_v^\lambda$-unit. Moreover,
$$
\delta(\Heis^n(\bb)) = \delta(\Heis^n(\bb),g_v^{\lambda_0}) = \sqrt{1 - \left( \frac{||\bb||_2}{||\bb||_4} \right)^4}
$$
where
$$
\lambda_0 = ||\bb||_2/||\bb||_4^2.
$$
More generally,
\begin{equation}\label{eq:delta-general}
\delta(\Heis^n(\bb),g_v^\lambda) = \sqrt{1 - 2\lambda^2 ||\bb||_2^2 + \lambda^4 ||\bb||_4^4}.
\end{equation}
The metric $g_v^\circ$ in \eqref{eq:gvo} is equal to $g_v^{\lambda_1}$ for some $\lambda_1>0$ which we will shortly determine. In fact, we have $g_h(\boldJ^{[g_h]}_\mu(X_j),Y_j) = \mu(b_j T) = b_j \lambda_1^2$ for all $j$. Hence $||\boldJ_\mu||_{HS} = \sqrt{2n} ||\bb||_2$ and
$$
\lambda_1 = ||\mu||_{g_v^\circ} = \tfrac1{\sqrt2} ||\boldJ_\mu||_{HS} = \sqrt{n} ||\bb||_2 \lambda_1^2
$$
whence we obtain
$$
\lambda_1 = (\sqrt{n} ||\bb||_2)^{-1}.
$$
The values $\lambda_0$ and $\lambda_1$ do not agree in general. In fact, if $\lambda_0 = \lambda_1$ then $||\bb||_4^2 = \sqrt{n} ||\bb||_2^2$, i.e.
\begin{equation}\label{eq:final2}
(\sum_{j=1}^n b_j^2)^2 = \sum_{j=1}^n b_j^4.
\end{equation}
Since $b_j>0$ for each $j=1,\ldots,n$, \eqref{eq:final2} holds if and only if $n=1$.
\end{remark}

\begin{remark}\label{rem:Euclidean-factors}
We conclude this section by discussing the behavior of the rank, M\'etivier order, and H-type deviation of a step two Carnot group with respect to the addition of horizontal abelian factors. 

For any Carnot group $\G$ and any integer $\nu$, the product group $\G \times \R^\nu$ is again equipped with the structure of a Carnot group. Here the Lie algebra $\Lie(\G\times \R^\nu) = W_1 \oplus W_2$, where $W_1 = V_1 \oplus \R^\nu$ and $W_2 = V_2$ and $\Lie(\G) = V_1 \oplus V_2$. In other words, the Euclidean factor $\R^\nu$ is viewed as part of the horizontal layer of the product group. Hence 
$$
\rank(\G\times\R^\nu) = \rank(\G) + \nu.
$$

Assume now that $\G$ has step two. Then $\G \times \R^\nu$ also has step two and
\begin{equation}\label{eq:Euclidean-factors}
\boldJ_\mu^{\G\times\R^\nu} = \begin{pmatrix} \boldJ_\mu^\G & \boldO \\ \boldO & \boldO \end{pmatrix}, \qquad \mu \in W_2^* = V_2^*.
\end{equation}
Consequently,
\begin{equation}\label{eq:order-with-Euclidean-factor}
\order(\Lie(\G\times\R^\nu)) = \order(\Lie(\G)).
\end{equation}
Equation \eqref{eq:Euclidean-factors} implies that
$$
||(\boldJ_\mu^{\G\times\R^\nu})^2 + \boldI||_{HS}^2 = ||(\boldJ_\mu^\G)^2 + \boldI||_{HS}^2 \qquad \forall \, \mu \in V_2^*.
$$
Hence, for a fixed metric $g_v$ on $V_2^*$,
\begin{equation*}\begin{split}
\delta(\G\times\R^\nu,g_v)^2 
&= \sup_{\mu:|\mu|_v = 1} \frac{||(\boldJ_\mu^{\G\times\R^\nu})^2 + \boldI||_{HS}^2}{\rank(\G\times\R^\nu)} \\
&= \frac{\sup_{\mu:|\mu|_v = 1} ||(\boldJ_\mu^\G)^2+\boldI||_{HS}^2 + \nu}{\rank(\G)+\nu} = \frac{\rank(\G) \delta(\G,g_v)^2 + \nu}{\rank(\G)+\nu}
\end{split}\end{equation*}
Infimizing over all vertical metrics $g_v$ yields
\begin{equation}\label{eq:H-type-deviation-of-G-times-R}
\delta(\G\times\R^\nu) = \sqrt{\frac{\rank(\G) \delta^2(\G) + \nu}{\rank(\G) + \nu}}.
\end{equation}
\end{remark}

Observe from \eqref{eq:H-type-deviation-of-G-times-R} and \eqref{eq:deviation-bounds} that $\delta(\G\times\R^\nu)>\delta(\G)$ whenever $\nu \ge 1$. Hence the inequality in \eqref{eq:H-type-deviation-vs-order} is consistent with the identity in \eqref{eq:order-with-Euclidean-factor}.

\medskip

In section \ref{sec:products} we compute in full generality the H-type deviation of arbitrary products of step two Carnot groups, including groups with horizontal abelian factors.

\section{Step two groups with small H-type deviation are M\'etivier groups}\label{sec:rigidity}

The notion of M\'etivier group originates from the study of analytic hypoellipticity of the sub-Laplacian in stratified groups. In \cite{metivier} the author proved that, in those groups, every homogeneous left-invariant hypoelliptic operator is analytically hypoelliptic. In the context of sub-Finsler geometry, M\'etivier groups can be characterized as those Carnot groups with no non-trivial singular length-minimizers, see \cite[Appendix A]{ln:regularity}. More information on this class of groups can be found in \cite[section 3.7]{blu} (where the terminology {\it H-group in the sense of M\'etivier} is used).

This section is devoted to the proof of Theorem \ref{th:h-type-rigidity}, which asserts that step two groups with sufficiently small H-type deviation are M\'etivier groups. We also refine the conclusion of the theorem by showing that step two groups with an even smaller value of the H-type deviation satisfy a quantitative version of the M\'etivier condition, which interpolates between that condition and the H-type condition.

\smallskip

\begin{proof}[Proof of Theorem \ref{th:h-type-rigidity}]
Let $\G$ be a step two Carnot group of rank $m$ with Lie algebra $\fg$. Towards the conclusion of part (a), assume first that $m$ is even and that $\delta(\G) < \sqrt{2/m}$. Proposition \ref{prop:order-vs-deviation} implies that
$$
\frac{m - \order(\fg)}{m} < \frac2m
$$
and hence $\order(\fg) > m-2$. Since both $\order(\fg)$ and $m$ are even, we must have $\order(\fg) = m$. Then $\fg$, and hence also $\G$, is M\'etivier.

We now turn to the proof of part (b). Arguing towards a contradiction, assume that $m$ is odd and that $\delta(\G) < \sqrt{1/m}$. Arguing as above we find that
$$
\frac{m - \order(\fg)}{m} < \frac1m
$$
and hence $\order(\fg) > m-1$. However, since $m$ is odd and $\order(\fg)$ is even we necessarily have $\order(\fg) \le m-1$. Thus no step two group with odd rank $m$ and H-type deviation less than $\sqrt{1/m}$ can exist.
\end{proof}

\begin{example}\label{ex:rigidity-sharpness}
For each even $m \ge 4$ there exists a step two Carnot group $\G$ of rank $m$ with $\delta(\G) = \sqrt{2/m}$ and such that $\G$ is not M\'etivier. For a concrete example, let $\G$ be the product Carnot group $\G = \Heis^n \times \R^2$. Since $m = \rank(\G) = 2n+2$ and $\Heis^n$ is an H-type group, the stated value for $\delta(\G)$ follows from the product formula \eqref{eq:H-type-deviation-of-G-times-R}.
%Let $V_1$ be an $m$-dimensional vector space. We introduce the relevant step two stratified Lie algebra $\fg = V_1 \oplus V_2$ by appealing to the alternate description in Remark \ref{rem:step-two-g-as-a-quotient-of-free-step-two}. Fix a subspace $\mathfrak{w}$ in $\%twowedge(V_1)$ so that $\mathfrak{w}^\perp$ is one-dimensional, spanned by a dual bi-vector $\mu$ such that $\boldJ_\mu$ has rank $m-2$. Then for any choice of a horizontal metric $g_h$ on $V_1$ we have $\delta(\G) \ge \sqrt{2/m}$, and equality holds for some %choice of horizontal metric. Moreover, since $\order(\fg) = m-2$, $\fg$ (and hence $\G$) is not M\'etivier.
\end{example}

We conclude this section by describing some more refined properties of step two groups with small H-type deviation. Specifically, we show that if such a group $\G$ has $\delta(\G)$ sufficiently close to zero, then $\G$ satisfies a quantitative and strengthened metric version of the M\'etivier condition. See Theorem \ref{th:H-type-rigidity-theorem} below.

Recall that a step two stratified Lie algebra $\fg = V_1 \oplus V_2$ is M\'etivier if the adjoint map $\ad_\bv:V_1 \to V_2$, $\ad_\bv(\bw) = [\bv,\bw]$, is surjective for each nonzero $\bv \in V_1$. The following characterization of groups of Heisenberg type is well known, see e.g.\ \cite[p.\ 148, condition (H)]{kap:h-type} or \cite[p.\ 3]{cdkr:h-type}.\footnote{Indeed, Kaplan \cite{kap:h-type} defined the class of H-type groups using condition (b) in Proposition \ref{prop:H-type-equivalence}.} Let us note that if $g_v$ is a fixed vertical metric on $V_2^*$, then $g_v$ induces by duality a canonical metric on $V_2$. For simplicity, by an abuse of notation, we also denote the dual metric on $V_2$ by $g_v$.

\begin{proposition}[Kaplan, Cowling--Dooley--Kor\'anyi--Ricci]\label{prop:H-type-equivalence}
Let $\G$ be a step two Carnot group equipped with horizontal metric $g_h$ and vertical metric $g_v$. Then the following are equivalent:
\begin{itemize}
\item[(a)] $\G$ is an H-type group.
\item[(b)] For each $\bv \in V_1$ with $|\bv|_h = 1$, $\ad_\bv:(\Ker(\ad_\bv)^\perp,g_h) \to (V_2,g_v)$ is an isometric isomorphism.
\end{itemize}
\end{proposition}

Here and in what follows, the orthocomplement $\Ker(\ad_\bv)^\perp$ is computed with respect to the horizontal metric $g_h$.

\begin{theorem}\label{th:H-type-rigidity-theorem}
Let $\G$ be a step two and rank $m$ Carnot group, and assume that
\begin{equation}\label{ex:H-type-rigidity-theorem}
\delta(\G,g_v) < \delta_0 < \frac1{\sqrt{m}}.
\end{equation}
for some vertical metric $g_v$. Then for each $\bv \in V_1$ with $|\bv|_h = 1$ and for all $\bw \in \Ker(\ad_\bv)^\perp$,
$$
A|\bw|_h \le |\ad_\bv(\bw)|_v \le B|\bw|_h.
$$
Here $A = (1-\sqrt{m}\delta_0)^{1/2}$ and $B = (1+\sqrt{m}\delta_0)^{1/2}$.
\end{theorem}

\begin{proof}[Proof of Theorem \ref{th:H-type-rigidity-theorem}]
Assume that $\delta(\G,g_v) < \delta_0$. Then
$$
|| (\boldJ_\mu^{[g_h]})^2 + |\mu|_v^2 \, \boldI ||_{HS} < \sqrt{m} |\mu|_v^2 \delta_0
$$
for all nonzero $\mu \in V_2^*$ (see e.g.\ \cite[(3.2)]{tys:h-type-stability}).

Let $\bv,\bw \in \fv_1$ and let $\mu \in V_2^*$ be nonzero. Then
\begin{equation*}\begin{split}
\left| |\mu|_v^2 \, g_h(\bv,\bw) - g_h(\boldJ_\mu^{[g_h]}(\bv),\boldJ_\mu^{[g_h]}(\bw)) \right|
&= \left| g_h(((\boldJ_\mu^{[g_v]})^2 + |\mu|_v^2 \, \boldI)(\bv),\bw) \right| \\
&\le ||(\boldJ_\mu^{[g_h]})^2 + |\mu|_v^2 \, \boldI||_{\op} \, |\bv|_h \, |\bw|_h \\
&\le ||(\boldJ_\mu^{[g_h]})^2 + |\mu|_v^2 \, \boldI||_{HS} \, |\bv|_h \, |\bw|_h \\
&< \sqrt{m} |\mu|_v^2 |\bv|_h |\bw|_h \delta_0.
\end{split}\end{equation*}
In particular,
\begin{equation}\label{eq:JTX-bi-Lip-bounds}
\left| |\mu|_v^2 \, |\bv|_h^2 - |\boldJ_\mu^{[g_h]}(\bv)|_h^2 \right| < \sqrt{m} |\mu|_v^2 |\bv|_h^2 \delta_0.
\end{equation}
If $\bv \in \fv_1$, $|\bv|_h = 1$, and $\bw \in \Ker(\ad_\bv)^\perp$, then
\begin{equation*}\begin{split}
|\ad_\bv(\bw)|_v 
&= \sup \left\lbrace g_v(\ad_\bv(\bw),\mu) : \mu \in V_2^*, |\mu|_v = 1 \right\rbrace \\
&= \sup \left\lbrace g_h(\boldJ_\mu^{[g_h]}(\bv),\bw) : \mu \in V_2^*, |\mu|_v = 1 \right\rbrace \\
&\le |\bw|_h \, \sup \left\lbrace |\boldJ_\mu^{[g_h]}(\bv)|_h : \mu \in V_2^*, |\mu|_v = 1 \right\rbrace \\
&\le (1+\sqrt{m}\delta_0)^{1/2} |\bw|_h
\end{split}\end{equation*}
by \eqref{eq:JTX-bi-Lip-bounds}.

\smallskip

For the other inequality, recall the identity $\Ker(\ad_\bv)^\perp = \cJ_{V_2^*}(\bv) := \{ \boldJ_\mu^{[g_h]}(\bv) \, : \, \mu \in V_2^* \}$, which holds in any step two Carnot group $\G$ (see, e.g., \cite[page 3]{cdkr:h-type}). Hence, for fixed $\bv \in V_1$ and $\bw \in \Ker(\ad_\bv)^\perp$ as above,
$$
\ad_\bv(\bw) = \sup \{ g_h(\boldJ_\mu^{[g_h]}(\bv),\bw) \, : \, \mu \in V_2^*, |\mu|_v = 1\}.
$$
For fixed $\bv \in V_1$ with $|\bv|_h = 1$ and $\bw \in \Ker(\ad_\bv)^\perp$, choose a dual vector $\mu_0 \in V_2^*$ so that $\boldJ_{\mu_0}^{[g_h]}(\bv) = \bw$. Without loss of generality we may assume that $\bw$ is nonzero; in this case we also have that $\mu_0$ is nonzero. Set
$$
\mu := \frac{\mu_0}{|\mu_0|_v}.
$$
Then $|\mu|_v = 1$ and hence
\begin{equation*}\begin{split}
|\ad_\bv(\bw)|_v 
&\ge g_h(\boldJ_\mu^{[g_h]}(\bv),\bw) = \frac1{|\mu_0|_v} g_h(\boldJ_{\mu_0}^{[g_h]}(\bv),\bw) \\
&= \frac{|\boldJ_{\mu_0}^{[g_h]}(\bv)|_h^2}{|\mu_0|_v} >  \frac{|\boldJ_{\mu_0}^{[g_h]}(\bv)|_h}{|\mu_0|_v} \, (1-\sqrt{m} \delta_0)^{1/2} |\mu_0|_v \\
\intertext{by \eqref{eq:JTX-bi-Lip-bounds}}
&=  (1-\sqrt{m} \delta_0)^{1/2} |\bw|_h.
\end{split}\end{equation*}
This completes the proof of Theorem \ref{th:H-type-rigidity-theorem}.
\end{proof}

\section{Vertical metrics realizing the H-type deviation}\label{sec:realization}

The main result of this section is the following theorem.

\begin{theorem}\label{th:attained}
Let $\G$ be a step two Carnot group with Lie algebra $\fg = V_1 \oplus V_2$. Then the infimum over vertical metrics in the definition of $\delta(\G)$ is attained, i.e., there exists a metric $g_{v,\infty}$ on $V_2^*$ so that
$$
\delta(\G) = \delta(\G,g_{v,\infty}).
$$
\end{theorem}

In \cite{tys:h-type-stability} this result was stated and proved in the special case $\delta(\G) = 0$. However, one step in that proof was not explained in sufficient detail. To wit, it was left unexplained why the limiting inner product $g_{v,\infty}$ was in fact a {\bf metric} rather than only a semimetric. 

Here we give a different proof for the fact that H-type deviation is realized by a given vertical metric, which handles all possible values for $\delta(\G)$ and avoids the above issue regarding the limiting object $g_{v,\infty}$. The key new ingredient is our improved upper bound $\delta(\G) \le \sqrt{1-2/\rank(\G)} < 1$ (Corollary \ref{cor:deviation-bounds}), which yields a compactness condition which in turn ensures that the limiting object is in fact a metric.

We comment further on the distinction between metrics and semimetrics in relation to the H-type deviation in Example \ref{ex:quaternionic-Heisenberg}.

\begin{proof}[Proof of Theorem \ref{th:attained}]
Let $\G$ be a Carnot group of step two and rank $m$, and let $(g_{v,j})$ be a sequence of vertical metrics so that
$$ 
\delta(\G,g_{v,j}) \le \delta(\G) + \frac{1}{j}.
$$
In view of Corollary \ref{cor:deviation-bounds} we can assume without loss of generality that 
\begin{equation}\label{eq:attained-upper-bound}
\delta(\G,g_{v,j}) < \sqrt{\frac{m-1}{m}} < 1
\end{equation}
for all $j$. We define the set
$$ 
K \coloneq \Set{ \mu \in V_2^* \, | \, \lVert \boldJ^2_{\mu}+\Id \rVert_{HS} \le \sqrt{m-1}}. 
$$
The set $K$ is compact and does not contain $0 \in V_2^*$. Moreover, \eqref{eq:attained-upper-bound} implies that the unit $g_{v,j}$ sphere, $\Sph_{g_{v,j}}$, satisfies
\begin{equation}\label{eq:spheres-in-K}
\Sph_{g_{v,j}} \subseteq K \qquad \forall \, j \in \N.
\end{equation}

Fix $j \in \N$. For each $0 \ne \mu \in V_2^*$ we have $(r_\mu^j)^{-1} \mu \in \Sph_{g_{v,j}}$ where $r_\mu^j = |\mu|_{v,j}$. Since $\mu \mapsto |\mu|_{v,j}$ is continuous, the map $\mu \mapsto (r_\mu^j)^{-1} \mu$ from $V_2^* \setminus \{0\}$ to $\Sph_{g_{v,j}}$ is a continuous retraction. Thus $\Sph_{g_{v,j}}$ is a topological and codimension one sphere in $V_2^*$ which separates $0$ from $\infty$. Letting $j \to \infty$ and appealing to a standard diagonalization and separability argument, we obtain a continuous map $\mu \mapsto r_\mu^\infty$ from $V_2^* \setminus \{0\}$ to $(0,\infty)$. Observe that \eqref{eq:spheres-in-K} ensures that $0<r_\mu^\infty<\infty$ for each nonzero $\mu$, since $K$ is compactly contained in $V_2^*\setminus \{0\}$.

It now follows that $|\mu|_{v,\infty} := r_\mu^\infty$ is a norm on $V_2^*$ which in turn induces a metric $g_{v,\infty}$ by polarization. To conclude the proof, choose $\mu \in V_2^*$ with $|\mu|_{v,\infty} = 1$ and define $\mu^j := (r_\mu^j)^{-1} \mu$ for each $j \in \N$. Then $\mu^j \to \mu$ and
\begin{equation*}\begin{split}
||\boldJ_\mu^2 + \boldI||_{HS} &= \lim_{j\to\infty} ||\boldJ_{\mu^j}^2 + \boldI||_{HS} \\
& \le \lim_{j\to\infty} \delta(\G,g_{v,j}) = \delta(\G).
\end{split}\end{equation*}
Hence
$$
\delta(\G,g_{v,\infty}) \le \delta(\G)
$$
and the proof is complete.
\end{proof}

In the following example, we illustrate the subtle issues related to the distinction between vertical metrics and vertical semimetrics in the definition of the H-type deviation. Recall that a {\it semimetric} on a vector space $V$ is a symmetric bilinear form $g$ which generates a seminorm $||\cdot||_g = g(\cdot,\cdot)^{1/2}$. That is, $||\cdot||_g$ satisfies all of the hypotheses of a norm, except that it may vanish on nonzero elements. 

It is natural to inquire whether the notion of H-type deviation remains the same if the infimum is expanded to range over all vertical semimetrics. For a vertical semimetric $g_v$ on $V_2^*$ we set 
\begin{equation}\label{eq:delta-for-vertical-semimetric}
\delta(\G,g_v) = \frac1{\sqrt{\dim\fv_1}} \, \sup_{\substack{\mu \in V_2^* \\ ||\mu||_v = 1}} ||(\boldJ_\mu)^2 + \Id||_{HS} \, ,
\end{equation}
where the supremum is taken over the lower-dimensional sphere in $V_2^*$ consisting of all dual vectors $\mu$ for which $g_v(\mu,\mu) = 1$. 

\begin{example}\label{ex:quaternionic-Heisenberg}
We focus attention on the first quaternionic Heisenberg group  $\G = \Heis^1_{\K}$. Following the presentation in \cite{km:quaternionic-Heisenberg}, we fix a $g_h$-orthonormal basis $\{X,Y,Z,W\}$ for the first layer $V_1$ of the Lie algebra $\fg$, and a spanning set $\{S,T,U\}$ for the second layer $V_2$ of $\fg$ satisfying the following nontrivial bracket relations:
\begin{equation*}
[X,Y] = S, \quad [Z,W] = S, \quad [X,Z] = T, \quad [Y,W] = -T, \quad [X,W] = U, \quad \mbox{and} \quad [Y,Z] = U.
\end{equation*}
The group $\G$ so defined is a step two Carnot group of rank four and dimension seven.

Let us denote by $\mu_S,\mu_T,\mu_U$ the basis for $V_2^*$ which is canonically dual to the basis $S,T,U$ for $V_2$. For convenience, we restrict our attention to vertical metrics $g_v$ on $V_2^*$ for which $\mu_S$, $\mu_T$, and $\mu_U$ are orthogonal; let us call such a metric {\it diagonal}. Vertical diagonal metrics are in one-to-one correspondence with triples $(a,b,c)$ of positive real numbers via the relations $g_v(\mu_S,\mu_S) = a$, $g_v(\mu_T,\mu_T) = b$, and $g_v(\mu_U,\mu_U) = c$. We compute the Kaplan operator in the basis $\{X,Z,Y,W\}$:
\begin{equation*}
\boldJ_{\mu_S} = \begin{pmatrix} 0 & a & 0 & 0 \\ -a & 0 & 0 & 0 \\ 0 & 0 & 0 & a \\ 0 & 0 & -a & 0 \end{pmatrix}, \qquad 
\boldJ_{\mu_T} = \begin{pmatrix} 0 & 0 & b & 0 \\ 0 & 0 & 0 & -b \\ -b & 0 & 0 & 0 \\ 0 & b & 0 & 0 \end{pmatrix}, \qquad 
\boldJ_{\mu_U} = \begin{pmatrix} 0 & 0 & 0 & c \\ 0 & 0 & c & 0 \\ 0 & -c & 0 & 0 \\ -c & 0 & 0 & 0 \end{pmatrix}.
\end{equation*}
Hence $\boldJ_{\mu_S}^2 = -a^2 \, \Id$,  $\boldJ_{\mu_T}^2 = -b^2 \, \Id$, and $\boldJ_{\mu_U} = -c^2 \, \Id$, and the matrices $\boldJ_{\mu_S}$, $\boldJ_{\mu_T}$ and $\boldJ_{\mu_U}$ anti-commute. Consequently, for any triple $s,t,u$,
$$
\boldJ_{\mu} = -(s^2a^2+t^2b^2+u^2c^2) \, \Id, \qquad \qquad \mu := s \, \mu_S + t \, \mu_T + u \, \mu_U.
$$
We calculate
\begin{equation}\begin{split}\label{eq:H1Kdeltav}
\delta(\Heis^1_{\K},g_v) &= \frac12 \sup_\mu \, ||\boldJ_\mu^2 + \Id ||_{HS} \\
&= \sup_{\substack{(s,t,u) \\ s^2+t^2+u^2=1}} |1-(s^2a^2+t^2b^2+u^2c^2)| \\
&= \sup_{\substack{(s,t,u) \\ s^2+t^2+u^2=1}} |(1-a^2)s^2 + (1-b^2)t^2 + (1-c^2)u^2| \\
&=\max\{|1-a^2|, |1-b^2|, |1-c^2| \}.
\end{split}\end{equation}
Note that $\delta(\Heis^1_{\K},g_v) = 0$ if and only if $a=b=c=1$: this choice of vertical metric equips $\Heis^1_{\K}$ with the structure of a (nascent) H-type group.

The preceding discussion shows that the H-type deviation parameter $\delta(\Heis^1_{\K},g_v)$ is positive for any choice of $a,b,c>0$ such that $(a,b,c) \ne (1,1,1)$. However, the situation is more nuanced if we allow some of the parameters $a,b,c$ to degenerate to zero, i.e., if we widen our view from vertical metrics to vertical semimetrics. Consider the vertical diagonal semimetric $g_{v,0}$ on $\fv$ for which $a=b=1$ and $c=0$. Thus $\mu_S$ and $\mu_T$ are $g_{v,0}$-unit vectors and $\mu_U$ is $g_{v,0}$-null. An easy exercise shows that $\delta(\G,g_{v,0}) = 0$, where $\delta(\G,g_{v,0})$ is defined as in \eqref{eq:delta-for-vertical-semimetric}.

It is natural to ask why this conclusion does not violate \cite[Proposition 3.2]{tys:h-type-stability}. Indeed, $g_{v,0}$ is a vertical semimetric with vanishing H-type deviation which does {\bf not} impose H-type structure on the Lie algebra.

The resolution of this apparent contradiction lies in the fact that this particular semimetric does not arise as the limit of a degenerating sequence of vertical metrics $g_{v,j}$ with H-type deviation tending to zero. For convenience we restrict attention to diagonal metrics. We appeal to the formula \eqref{eq:H1Kdeltav} for the H-type deviation of a general vertical diagonal metric. Assume that $g_{v,j}$ is a sequence of vertical diagonal metrics such that $g_{v,j} \to g_{v,0}$ as $j\to\infty$. If we identify $g_{v,j}$ with a triple $(a_j,b_j,c_j)$ as before, this convergence corresponds to the assumptions $a_j \to 1$, $b_j \to 1$, and $c_j \to 0$. Then
$$
\delta(\Heis^1_{\K},g_{v,j}) = \max\{|1-a_j^2|, |1-b_j^2|, |1-c_j^2| \} \to 1 \ne 0 = \delta(\Heis^1_{\K},g_{v,0}).
$$
A similar phenomenon was identified in Example 3.7 of \cite{tys:h-type-stability}.
\end{example}

\begin{example}\label{ex:collapsed-quaternionic-Heisenberg-group}
In contrast with the preceding example, consider the group $\G$ of rank four and dimension six, whose horizontal layer $V_1$ admits the same $g_h$-orthonormal basis $\{X,Y,Z,W\}$, and whose vertical layer $V_2$ admits the basis $\{S,T\}$. The bracket relations are assumed to be the same ones as before, i.e.,  $[X,Y] = S$, $[Z,W] = S$, $[X,Z] = T$, and $[Y,W] = T$. In this example, the vertical diagonal metric $g_v$ which makes the pair $\mu_S,\mu_T$ orthonormal realizes the H-type deviation, which in this case is equal to zero. In fact, this particular group is an H-type group. Compare \cite[Example~3.6]{bnv:SP2}.
\end{example}

\begin{example}
In an earlier version of this paper ({\tt arxiv.org/abs/2312.06076}), the existence of a vertical metric realizing the value of the H-type deviation was obtained, in the special case $\delta(\G) = 0$, by a different method. To wit, pointwise bounds were established for a sequence of vertical metrics which achieve the infimum in the definition of $\delta(\G)$; these pointwise bounds in turn imply the existence of a limiting semimetric. After that, Theorem \ref{th:H-type-rigidity-theorem} was used to conclude that the limiting object $g_{v,\infty}$ satisfied condition (b) in Proposition \ref{prop:H-type-equivalence}. The H-type property of $\G$ follows in turn from that proposition. As noted above, we provide here an alternative argument which bypasses the need to consider semimetrics and directly establishes the existing of the limiting metric via a compactness condition originating from our improved upper bound on the values of the H-type deviation for step $m$ groups.
\end{example}

\section{H-type deviation of direct products of step two Carnot groups}\label{sec:products}

In this section we prove the product formulas \eqref{eq:products-h-type-deviation} and \eqref{eq:products-h-type-deviation-with-Euclidean-factors}. In Example \ref{ex:products-anisotropic-heisenberg}, we illustrate this formula by stating the value of $\delta(\G)$ whenever $\G = \G_1\times \cdots \times \G_\ell \times \R^\nu$ is a product of Heisenberg groups (including anisotropic Heisenberg groups) and with possible horizontal abelian factors.

The identity \eqref{eq:products-h-type-deviation-with-Euclidean-factors} follows from \eqref{eq:products-h-type-deviation} and Remark \ref{rem:Euclidean-factors}. Hence we focus our attention on \eqref{eq:products-h-type-deviation}, which we restate here for the convenience of the reader:
\begin{equation}\label{eq:products-h-type-deviation-2}
	\delta(\G_1 \times \cdots \times \G_\ell)= \sqrt{1- \frac{\min_i\set{\rank(\G_i)(1-\delta^2(\G_i))}}{\rank(\G_1\times\cdots\times\G_\ell)}}.
\end{equation}
The following reformulation of \eqref{eq:products-h-type-deviation-2} illustrates the connection to the M\'etivier order of the Lie algebras of the product and factor groups:
$$
\rank(\G) \, (1-\delta^2(\G)) = \min_i \rank(\G_i) \, (1-\delta^2(\G_i))\,, \qquad \qquad \G = \G_1\times\cdots\times\G_\ell.
$$
Observe that $\rank(\G) \, (1-\delta^2(\G)) \le \order(\G)$ for any step two group $\G$, cf.\ Proposition \ref{prop:order-vs-deviation}.

\begin{proof}[Proof of Theorem \ref{th:products-h-type-deviation}]
We first observe that the desired claim is consistent with the associativity rule
$$
\delta(\G_1 \times \G_2 \times \G_3)=\delta((\G_1 \times \G_2)\times \G_3).
$$ 
It thus suffices to prove the result for $\ell=2$. 

Fix vertical metrics $g^1_v$ and $g^2_v$ on $\G_1$ and $\G_2$ respectively. We define $m_1=\rank(\G_1)$ and $m_2=\rank(\G_2)$ and $\delta_1=\delta(\G_1,g_v^1)$ and $\delta_2=\delta(\G_2,g^2_v)$. We now estimate $\delta \coloneq \delta(\G_1 \times \G_2, g^1_v \otimes g^2_v)$.
\begin{equation*}\begin{split}
\delta &= \frac{1}{\sqrt{m_1+m_2}}\sup_{s_1^2+s_2^2=1}\left\lVert \boldJ^2_{s_1\mu_1+s_2\mu_2}+\boldI\right\rVert_{HS} \\
&= \sup_{s_1^2+s_2^2=1}\sqrt{\frac{\left\lVert \boldJ^2_{s_1\mu_1}+\boldI_{V_1}\right\rVert^2_{HS}+\left\lVert \boldJ^2_{s_2\mu_2}+\boldI_{V_2}\right\rVert_{HS}^2}{m_1+m_2}}\\
&=\sup_{s_1^2+s_2^2=1}\sqrt{\frac{\left\lVert s_1^2\boldJ^2_{\mu_1}+s_1^2\boldI_{V_1}+s_2^2\boldI_{V_1}\right\rVert^2_{HS}+\left\lVert s_2^2\boldJ^2_{\mu_1}+s_2^2\boldI_{V_2}+s_1^2\boldI_{V_2}\right\rVert^2_{HS}}{m_1+m_2}}\\
&\le \sup_{s_1^2+s_2^2=1}\sqrt{\frac{s_1^4m_1\delta_1^2+s_2^4m_1+2s_1^2s_2^2m_1\delta_1+s_2^4m_2\delta_2+s_1^4m_2+2s_1^2s_2^2m_2\delta_2}{m_1+m_2}} \\
&\stackrel{\lambda \coloneq s_1^2}{=} \sup_{0 \le \lambda \le 1} \sqrt{\frac{\lambda^2(m_1\delta_1^2 + m_2)+(1-\lambda)^2(m_2\delta_2^2+m_1)+2\lambda(1-\lambda)(m_1\delta_1+m_2\delta_2)}{m_1+m_2}}
\end{split}\end{equation*}
It is immediate to check that the numerator inside the root is a convex expression in $\lambda$, hence obtains its supremum on the boundary $\lambda \in \set{0,1}$. Thus
\begin{equation*}\begin{split}
\delta &\le \sqrt{\frac{\max \set{m_1\delta^2_1+m_2,m_2\delta^2_2+m_1}}{m_1+m_2}}=\sqrt{\frac{m_1+m_2 + \max\set{m_1\delta^2_1-m_1,m_2\delta^2_2-m_2}}{m_1+m_2}}\\
&=\sqrt{\frac{m_1+m_2 - \min\set{m_1(1-\delta^2_1),m_2(1-\delta^2_2)}}{m_1+m_2}}=\sqrt{1- \frac{\min\set{m_1(1-\delta^2_1),m_2(1-\delta^2_2)}}{m_1+m_2}} \, .
\end{split}\end{equation*}
This completes the proof that $\delta(\G)$ is bounded above by the expression on the right hand side of \eqref{eq:products-h-type-deviation-2}.

We now turn to the proof of the other inequality. Fix a vertical norm $g_v$ for $\G_1 \times \G_2$. We define $g_v^1$ and $g_v^2$ to be the restrictions of $g_v$ on $\G_1$ and $\G_2$, respectively. For $\mu$ in the dual of the second layer of $\G_1 \times \G_2$, we write $\mu = \mu_1 + \mu_2$, with $\mu_1$ and $\mu_2$ in the dual of the vertical layer of $\G_1$ and $\G_2$, respectively. Then 
\begin{equation*}\begin{split}
\delta(\G_1\times \G_2,g_v)&= \frac{1}{\sqrt{m_1+m_2}} \sup_{\lvert \mu_1 + \mu_2\rvert_v=1} \lVert \boldJ^2_{\mu_1+\mu_2}+\boldI \rVert_{HS} \\
&\ge \frac{1}{\sqrt{m_1+m_2}} \sup_{\lvert \mu_1 \rvert_v=1} \lVert \boldJ^2_{\mu_1}+\boldI \rVert_{HS} \\
&\ge  \sup_{\lvert \mu_1 \rvert_v=1} \sqrt{\frac{\lVert \boldJ^2_{\mu_1}+\boldI_{V_1} \rVert^2_{HS}+ \lVert \boldI_{V_2} \rVert^2_{HS}}{m_1+m_2}} \\
&\ge \sqrt{\frac{m_1\delta^2(\G_1)+m_2}{m_1+m_2}}.
\end{split}\end{equation*}
With a similar argument we prove that
\begin{equation*}
	\delta(\G_1\times \G_2,g_v) \ge \sqrt{\frac{m_2\delta^2(\G_2)+m_1}{m_1+m_2}}.
\end{equation*}
Therefore
\begin{equation*}\begin{split}
	\delta(\G_1\times \G_2,g_v) &\ge \sqrt{\frac{\max \set{m_1\delta^2(\G_1)+m_2,m_2\delta^2(\G_2)+m_1}}{m_1+m_2}} \\&= \sqrt{1- \frac{\min\set{m_1(1-\delta^2(\G_1)),m_2(1-\delta^2(\G_2))}}{m_1+m_2}}.
\end{split}\end{equation*}
Since the estimate holds for every vertical norm, the proof is complete.
\end{proof}

\begin{example}\label{ex:products-anisotropic-heisenberg}
Using the formula for the H-type deviation of the anisotropic Heisenberg group $\Heis^n(\bb)$ in Example \ref{ex:anisotropic-heisenberg}, we derive the value of H-type deviation for products of such groups along with possible horizontal abelian factors.

Let $k_1,\ldots,k_\ell$ be positive integers with $k_1+\cdots+k_\ell = n$, and let $\bb_1,\ldots,\bb_\ell$ be a collection of vectors of positive real numbers. Assume without loss of generality that the pairs $(k_i,\bb_i)$ are ordered so that
\begin{equation}\label{eq:ks}
k_1 \left(\frac{||\bb_1||_2}{||\bb_1||_4}\right)^4 \le \cdots \le k_\ell \left(\frac{||\bb_\ell||_2}{||\bb_\ell||_4}\right)^4.
\end{equation}
Then
\begin{equation}\label{eq:h-type-deviation-of-products-of-anisotropic-Heisenberg-groups}
\delta(\Heis^{k_1}(\bb_1) \times \cdots \times \Heis^{k_\ell}(\bb_\ell)) = \sqrt{1 - \frac{k_1}{n} \, \left(\frac{||\bb_1||_2}{||\bb_1||_4}\right)^4 }.
\end{equation}
In particular, when $\Heis^n(\bb)$ is an isotropic Heisenberg group, $\bb$ is a constant vector and the norms $||\bb||_p$ are constant in $p$. Thus if $1\le k_1 \le \cdots \le k_\ell$ and $n=k_1 + \cdots + k_\ell$, then 
\begin{equation}\label{eq:h-type-deviation-of-products-of-isotropic-Heisenberg-groups}
\delta(\Heis^{k_1} \times \cdots \times \Heis^{k_\ell}) = \sqrt{1 - \frac{k_1}{n}}.
\end{equation}
Similarly,
\begin{equation}\label{eq:h-type-deviation-of-products-of-anisotropic-Heisenberg-groups-with-Euclidean-factors}
\delta(\Heis^{k_1}(\bb_1) \times \cdots \times \Heis^{k_\ell}(\bb_\ell) \times \R^\nu) = \sqrt{1 - \frac{2k_1}{2n+\nu} \, \left(\frac{||\bb_1||_2}{||\bb_1||_4}\right)^4 }.
\end{equation}
and
\begin{equation}\label{eq:h-type-deviation-of-products-of-isotropic-Heisenberg-groups-with-Euclidean-factors}
\delta(\Heis^{k_1} \times \cdots \times \Heis^{k_\ell} \times \R^\nu) = \sqrt{1 - \frac{2k_1}{2n+\nu}}.
\end{equation}
\end{example}

We conclude this paper by commenting further on the case when $\delta(\G)$ is maximal among groups with a given rank. 

\begin{proposition}\label{prop:h-type-deviation-of-products-of-anisotropic-Heisenberg-groups-maximality-condition}
Let $\G = \G_1 \times \cdots \times \G_\ell$ be a product of step two Carnot groups, with $m=\rank(\G)$ and $m_i = \rank(\G_i)$ for each $i$. 

Then $\delta(\G) = \sqrt{(m-2)/m}$ if and only if $\delta(\G_i) = \sqrt{(m_i-2)/m_i}$ for some $i$.

In particular, any such group of the form $\G = \K \times \Heis^1$ satisfies $\delta(\G) = \sqrt{(m-2)/m}$.
\end{proposition}

Proposition \ref{prop:h-type-deviation-of-products-of-anisotropic-Heisenberg-groups-maximality-condition} is an easy consequence of the product formula \eqref{eq:products-h-type-deviation-2}. Alternatively, it follows from the optimality criterion Corollary \ref{cor:deviation-bounds-optimality} and the elementary observation that $\order(\Lie(\G)) = 2$ if and only if $\order(\Lie(\G_i)) = 2$ for some $i$. We state a simple corollary.

\begin{corollary}\label{cor:h-type-deviation-of-products-of-anisotropic-Heisenberg-groups-maximality-condition}
Let $\G$ be a product of anisotropic Heisenberg groups, and assume that $\G$ has rank $m$. Then $\delta(\G) = \sqrt{(m-2)/m}$ if and only if one of the factor groups is the first isotropic Heisenberg group $\Heis^1$.
\end{corollary}

In view of Proposition \ref{prop:h-type-deviation-of-products-of-anisotropic-Heisenberg-groups-maximality-condition} and Corollary \ref{cor:deviation-bounds-optimality}, it suffices to observe that an anisotropic Heisenberg group $\G = \Heis^n(\bb)$ has M\'etivier order two if and only if $n=1$. Moreover, the groups $\Heis^1(b)$, $b>0$, are all mutually Carnot group isomorphic, and in particular are Carnot group isomorphic to the isotropic group $\Heis^1$.

%\begin{proof}[Proof of Corollary \ref{cor:h-type-deviation-of-products-of-anisotropic-Heisenberg-groups-maximality-condition}]
%In view of the proposition, it suffices to prove that
%\begin{equation}\label{eq:H-n-b}
%\delta(\Heis^n(\bb)) = \sqrt{ \frac{2n-2}{2n} }
%\end{equation}
%if and only if $n=1$ and the group is the isotropic Heisenberg group. If \eqref{eq:H-n-b} holds, then from the formula in Example \ref{ex:anisotropic-heisenberg} we conclude that $|| \bb ||_4^4 = n \, ||\bb||_2^4$ i.e.,
%$$
%\sum_{j=1}^{n} (b_{j})^4 = \left( \sum_{j=1}^{n} (b_{j})^2 \right)^2, \qquad \bb = (b_{1},\ldots,b_{n}).
%$$
%Since $b_{j}>0$ for all $j$, we conclude that $n = 1$. 
%\end{proof}

\bibliographystyle{acm}
\bibliography{biblio}

\begin{thebibliography}{10}

\bibitem{bt:polar}
{\sc Balogh, Z.~M., and Tyson, J.~T.}
\newblock Polar coordinates in {C}arnot groups.
\newblock {\em Math.\ Z. 241\/} (2002), 697--730.

\bibitem{br:MCP}
{\sc Barilari, D., and Rizzi, L.}
\newblock Sharp measure contraction property for generalized {H}-type {C}arnot
  groups.
\newblock {\em Commun. Contemp. Math. 20}, 6 (2018), 1750081, 24.

\bibitem{bg:CDK}
{\sc Baudoin, F., and Garofalo, N.}
\newblock Curvature-dimension inequalities and {R}icci lower bounds for
  sub-{R}iemannian manifolds with transverse symmetries.
\newblock {\em J. Eur. Math. Soc. (JEMS) 19}, 1 (2017), 151--219.

\bibitem{bnv:SP2}
{\sc Boarotto, F., Nalon, L., and Vittone, D.}
\newblock The {S}ard problem in step 2 and in filiform {C}arnot groups.
\newblock {\em ESAIM Control Optim. Calc. Var. 28\/} (2022), Paper No. 75, 20.

\bibitem{blu}
{\sc Bonfiglioli, A., Lanconelli, E., and Uguzzoni, F.}
\newblock {\em Stratified {L}ie groups and potential theory for their
  sub-{L}aplacians}.
\newblock Springer Monographs in Mathematics. Springer, Berlin, 2007.

\bibitem{cdg:carnot}
{\sc Capogna, L., Danielli, D., and Garofalo, N.}
\newblock Capacitary estimates and the local behavior of solutions of nonlinear
  subelliptic equations.
\newblock {\em Amer.\ J. Math. 118\/} (1996), 1153--1196.

\bibitem{cdkr:h-type}
{\sc Cowling, M., Dooley, M., Kor\'anyi, A., and Ricci, F.}
\newblock H-type groups and {I}wasawa decompositions.
\newblock {\em Adv.\ Math.\ 87\/} (1991), 1--41.

\bibitem{eld:H-type-2}
{\sc Eldredge, N.}
\newblock Precise estimates for the subelliptic heat kernel on {$H$}-type
  groups.
\newblock {\em J. Math. Pures Appl. (9) 92}, 1 (2009), 52--85.

\bibitem{eld:H-type-1}
{\sc Eldredge, N.}
\newblock Gradient estimates for the subelliptic heat kernel on {$H$}-type
  groups.
\newblock {\em J. Funct. Anal. 258}, 2 (2010), 504--533.

\bibitem{kap:h-type}
{\sc Kaplan, A.}
\newblock Fundamental solutions for a class of hypoelliptic {P}{D}{E} generated
  by composition of quadratic forms.
\newblock {\em Trans.\ Amer.\ Math.\ Soc. 258}, 1 (1980), 147--153.

\bibitem{km:quaternionic-Heisenberg}
{\sc Korolko, A., and Markina, I.}
\newblock Geodesics on {$\Bbb H$}-type quaternion groups with sub-{L}orentzian
  metric and their physical interpretation.
\newblock {\em Complex Anal.\ Oper.\ Theory 4}, 3 (2010), 589--618.

\bibitem{ln:regularity}
{\sc Le~Donne, E., and Nicolussi~Golo, S.}
\newblock Regularity properties of spheres in homogeneous groups.
\newblock {\em Trans. Amer. Math. Soc. 370}, 3 (2018), 2057--2084.

\bibitem{mv:quasimeromorphic-type-H}
{\sc Markina, I., and Vodopyanov, S.}
\newblock On value distributions for quasimeromorphic mappings on {$\Bbb
  H$}-type {C}arnot groups.
\newblock {\em Bull. Sci. Math. 130}, 6 (2006), 467--523.

\bibitem{metivier}
{\sc M\'etivier, G.}
\newblock Hypoellipticit\'e{} analytique sur des groupes nilpotents de rang
  {$2$}.
\newblock {\em Duke Math. J. 47}, 1 (1980), 195--221.

\bibitem{ms:maximal-operators}
{\sc M\"{u}ller, D., and Seeger, A.}
\newblock Singular spherical maximal operators on a class of two step nilpotent
  {L}ie groups.
\newblock {\em Israel J. Math. 141\/} (2004), 315--340.

\bibitem{rig:mass-transportation-type-H}
{\sc Rigot, S.}
\newblock Mass transportation in groups of type {$H$}.
\newblock {\em Commun. Contemp. Math. 7}, 4 (2005), 509--537.

\bibitem{tys:polar2}
{\sc Tyson, J.~T.}
\newblock Polar coordinates in {C}arnot groups {I}{I}.
\newblock To appear in a volume of the AMS series {\it Contemp.\ Math.}

\bibitem{tys:h-type-stability}
{\sc Tyson, J.~T.}
\newblock Stability theorems for {H}-type {C}arnot groups.
\newblock {\em J. Geom.\ Analysis 33\/} (2023), Article number: 329.

\end{thebibliography}

\end{document}